\documentclass{amsart}
\usepackage{amsmath,amsthm,amssymb,IMjournal}

\newcommand{\R}{\mathbb R}

\renewcommand{\geq}{\geqslant}
\renewcommand{\leq}{\leqslant}

\renewcommand{\v}{\vert}

\def\D{\Delta_x}
\def\Dx{\Delta_x}
\def\Cal{\mathcal}
\def\({\left(}
\def\){\right)}
\def\Nx{\nabla_x}
\def\Dt{\partial_t}

\def\Bbb{\mathbb}

\begin{document}
\newtheorem{The}{Theorem}[section]
\newtheorem{Lem}{Lemma}[section]
\newtheorem{Def}{Definition}[section]
\newtheorem{Cor}{Corollary}[section]
\newtheorem{Rem}{Remark}[section]

\numberwithin{equation}{section}

\title{Recent progress in attractors for quintic wave
equations}

\author{\|Anton |Savostianov|, Guildford,
        \|Sergey |Zelik|, Guildford}



\abstract
   We report on new results concerning the global well-posedness, dissipativity and attractors of the damped quintic wave equations in bounded domains of $\R^3$.
\endabstract

\keywords
   damped wave equation, fractional damping, critical nonlinearity, global attractor, smoothness
\endkeywords

\subjclass
35B40, 35B45
\endsubjclass

\thanks
   This work is partially supported by the Russian Ministry of Education and Science (contract no.
8502).
\endthanks

\section{Introduction}\label{sec1}
We consider the following non-linear damped wave equation
\begin{equation}
\label{eq dw}
\begin{cases}
\Dt ^2u+\gamma(-\D)^\theta\Dt u -\D  u+f(u)=g(x),\\
u|_{t=0}=u_0,\quad \Dt u|_{t=0}=u'_0
\end{cases}
\end{equation}
in a bounded smooth domain $\Omega\subset\R^3$ endowed with Dirichlet boundary conditions. Here $\D$ is the Laplacian with respect to the variable $x=(x^1,x^2,x^3)$, $\theta\in[0,1]$ and $\gamma>0$ are  given exponents, $g\in L^2(\Omega)$ is a given external force and $f$ is a nonlinearity which satisfy some natural dissipativity and growth assumptions, say,
\begin{equation}\label{0.fext}
-C+\kappa |u|^q\le f'(u)\le C(1+|u|^q)
\end{equation}
for some positive $C$ and $q$.
\par
Wave equations of the form \eqref{eq dw} are of big  interest from both theoretical and applied points of view and has been studied by many authors, see \cite{BV, CV, carc1, carc2, carc3, Chu2010, KZwvEq2009, PataZel2006 rem, PataZel2006,tri1,tri2} and references therein. Remarkable
is that even on the linear level ($f=g=0$), these equations
demonstrate rather non-trivial analytic properties in a strong  dependence  on the value of the exponent $\theta$. For the convenience of the reader, we briefly summarize them in the following table:
  \begin{center}
	\begin{tabular}{|c|c|c|c|}
	\hline
	$\theta$ & Semigroup & Smoothing  & Maximal Regularity \\
	\hline
	0 & $C_0$ & Asymptotic & No \\
	\hline
	$\ (0,\frac{1}{2}) \ $ & $C^\infty$ & Instantaneous & No\\
	\hline
	$\ [\frac{1}{2},1) \ $ & Analytic & Instantaneous & Yes \\
	\hline
	$1$ & Analytic & \mbox{ }Instantaneous for $\Dt u$,\mbox{ } & Yes\\
	& & asymptotic for $u$ & \\ \hline
	\end{tabular}
	\end{center}
See \cite{carc3,PataZel2006, tri1,tri2} for more details. As we can see from this table, there are three important borderline cases:
the first one ($\theta=0$) corresponds to the classical (weakly) damped wave equation, the second one ($\theta=1$) gives the so-called {\it strongly damped} wave equations and the third one ($\theta=\frac12$) is often referred to as wave equation with {\it structural} damping although the intermediate choices of $\theta$ are also interesting, see e.g., \cite{carc3,KZwvEq2009,PataZel2006} and references therein.
\par
The situation becomes much more delicate in the presence of the non-linearity	$f$ since  the analytic properties of solutions start to depend also on the growth rate of $f(u)$ as $u\to\infty$ (on the exponent $q$ in \eqref{0.fext}). Recall that the solutions $u(t)$ of problem \eqref{eq dw} satisfy (at least formally) the following energy identity:
\begin{equation}\label{0.e}
\frac d{dt} E(u(t),\Dt u(t))=-\gamma\|(-\Dx)^{\theta/2}\Dt u(t)\|^2_{L^2},
\end{equation}
where $E(u,v)=\frac12\|\Dt v\|^2_{L^2}+\frac12\|\Nx u\|^2_{L^2}+(F(u),1)-(g,u)$. Here and below $(u,v)$ stands for the usual inner product in $L^2(\Omega)$ and $F(u):=\int_0^uf(v)\,dv$ is a potential of the nonlinearity $f$. Thus, a weak {\it energy} solution of problem \eqref{eq dw} on the interval $t\in[0,T]$ is naturally defined as a function $u(t)$ which has the following regularity:
\begin{multline}
u\in L^\infty(0,T;H^1_0(\Omega)\cap L^{q+2}(\Omega)),\ \Dt u\in L^\infty(0,T;L^2(\Omega)),\\ (-\Dx)^{\theta/2}\Dt u\in L^2(0,T;L^2(\Omega))
\end{multline}
(here $H^s(\Omega)$ stands for the usual Sobolev space of distributions whose derivatives up to order $s$ belong to $L^2$ and $H^s_0(\Omega)$ is the closure of $C^\infty_0(\Omega)$ in $H^s(\Omega)$)
and satisfies equation \eqref{eq dw} in the sense of distributions. The corresponding energy phase space~is
\begin{equation}\label{0.energy}
\Cal E:=[H^1_0(\Omega)\cap L^{q+2}(\Omega)]\times L^2(\Omega),\ \ \xi_u:=(u,\Dt u)\in\Cal E.
\end{equation}
Recall also that, due to the Sobolev embedding $H^1_0\subset L^6$, we can take $\Cal E=H^1_0(\Omega)\times L^2(\Omega)$ if the growth exponent $q\le 4$ (in particular, for the case of quintic nonlinearities), but the term $L^{q+2}$ in the definition \eqref{0.energy} of the energy phase space looks unavoidable if the growth exponent $q>4$. The next standard result shows that weak energy solutions exist and are dissipative for all admissible values of $\theta$ and $q$.
\begin{The}\label{Th0.ensol} Let $\theta\in[0,1]$, $q\ge0$, $g\in L^2(\Omega)$, the nonlinearity $f$ satisfy \eqref{0.fext} and the initial data
$\xi_u(0):=(u_0,u_0')\in\Cal E$. Then, there exists at least one weak energy solution $u(t)$ defined for all $t\ge0$ which satisfies the following estimate:
\begin{equation}\label{0.enest}
\|\xi_u(t)\|_{\Cal E}+\int_t^{t+1}\|(-\Dx)^{\theta/2}\Dt u(s)\|^2_{L^2}\,ds\le Q(\|\xi_u(0)\|_{\Cal E})e^{-\alpha t}+Q(\|g\|_{L^2}),
\end{equation}
where the positive constant $\alpha$ and monotone function $Q$ are independent of $u$ and $t$.
\end{The}
The proof of this result is straightforward. Indeed, the dissipative estimate \eqref{0.enest} formally follows by multiplication of equation \eqref{eq dw} by $\Dt u+\beta u$ for properly chosen positive constant $\beta$ followed by the Gronwall's inequality and the existence of a solution can be verified, say, by Galerkin approximations, see \cite{BV,CV}.
\par
In contrast to this, the uniqueness and further regularity of energy solutions is more difficult and requires  essential restrictions on the growth exponent $q$. To the best of our knowledge this problem has been solved before only for the values of $\theta$ and $q$ collected in the following table:
\begin{center}
\begin{tabular}{|c|c|c|c|c|c|}
\hline
$\ \theta \ $  & $0$ & $(0,\frac{1}{2})$ & $\frac{1}{2}$ & $(\frac{1}{2},\frac{3}{4})$ & $[\frac{3}{4},1]$ \\
\hline
$\ q\ \ $ & $[0,2]\ $ & ? & $\ [0,4)\ $ & $\ 0\leq\ q(\theta)<\frac{8\theta}{3-4\theta}\ $ & $\ (0,\infty)\ $ \\
\hline
\end{tabular}\\[0.2cm]
\end{center}
See \cite{BV,carc3,PataZel2006,KZwvEq2009} for more details. E.g. for the case of structural damping $\theta=1/2$, the open problem was to treat the "critical" case of quintic non-linearities $q=4$ and, for the classical case $\theta=0$ in bounded domains with Dirichlet boundary conditions the theory has been developed only for cubic and sub-cubic nonlinearities ($q\le2$) although $q=4$ has been conjectured as the critical exponent here. For domains without boundary: $\Omega=\R^3$ or $\Omega=\Bbb T^3$ (periodic boundary conditions), a reasonable theory exists for $q<4$ in \cite{feireisl,kap}. In the quintic case ($\theta=0$, $q=4$) and $\Omega=\R^3$, the global existence of regular solutions has not been known for a long time (see \cite{SS,Grill,kap1}), but the attractor theory has been not developed for this case.

The aim of these notes is to present our recent results concerning global well-posedness, dissipativity and existence of smooth global attractors for quintic ($q=4$) wave equations in bounded domains, see \cite{ASqdw} and \cite{KSZ} for a detailed exposition. We restrict ourselves to discuss only two borderline cases $\theta=\frac12$ and $\theta=0$ although our conjecture is that the analogous results hold for all $\theta\in(0,\frac12)$.
Note that in both cases, the regularity of solutions provided by the energy estimate is not sufficient and the results are obtained by verifying some extra space-time regularity of the solutions although how this regularity is obtained is very different: in the case of structural damping, we derive an extra Lyapunov type estimate based on the multiplication of equation \eqref{eq dw} by $(-\Dx)^{1/2}u$ and in the weakly damped case it is achieved by utilizing the so-called Strichartz type estimate which are recently extended to the case of bounded domains, see \cite{stri,sogge1}. We discuss each of these two cases in more details below.
\section{Quintic wave equation: the case of structural damping $\theta=\frac12$}\label{sec3}
The key novelty here is the following theorem proved in \cite{ASqdw}.
\begin{The}\label{th main1} Let $\theta=\frac12$, $g\in L^2(\Omega)$ and the nonlinearity $f(u)$ be odd and satisfy \eqref{0.fext} with $q=4$. Then, any weak energy solution $u(t)$ of problem \eqref{eq dw} belongs to the space $L^2(0,T;H^{3/2}(\Omega))$ and satisfies the estimate
\begin{equation}\label{1.extra}
\|u\|_{L^2(t,t+1;H^{3/2}(\Omega))}\le Q(\|\xi_u\|_{L^\infty(t,t+1;\Cal E)})+Q(\|g\|_{L^2}),
\end{equation}
where the monotone function $Q$ is independent of $t$ and $u$.
\end{The}
\begin{proof}
We sketch the proof of this theorem for the case of periodic boundary conditions $\Omega=\Bbb T^3$. The case of Dirichlet boundary conditions can be reduced to that one using the odd extension of the solution $u$ through the boundary (to this end, we need the assumption that the nonlinearity $f$ is odd) together with a proper cut-off procedure, see \cite{ASqdw} for the details. To this end, we use the following identity obtained by multiplying \eqref{eq dw} by $(-\Dx)^{1/2}u$:
\begin{multline*}
\frac d{dt}\((\Dt u,(-\Dx)^{1/2}u)+\frac\gamma2\|(-\Dx)^{1/2}u\|^2_{L^2}\)+\|(-\Dx)^{3/4}u\|^2_{L^2}+\\+(f(u),(-\Dx)^{1/2}u)=(g,(-\Dx)^{1/2}u)+\|(-\Dx)^{1/4}\Dt u\|^2_{L^2}.
\end{multline*}
The integration of this identity over the interval $(t,t+1)$ gives the desired estimate \eqref{1.extra} in a straightforward way if we are able to estimate the term containing the nonlinearity $f$. To this end, we utilize the following lemma which can be verified using Fourier series and
the Parseval equality.
\begin{Lem}
Let $s\in(0,1)$ and $u$, $v\in H^s(\Bbb{T}^3)$, then
\begin{equation}
(v,(-\D)^su){\bf=}c\int_{\R^3}\int_{\Bbb{T}^3}\frac{(v(x+h)-v(x))(u(x+h)-u(x))}{|h|^{3+2s}}dxdh,
\end{equation}
where the constant $c$ depends only on $s$.
\end{Lem}
This lemma, together with the assumption $f'(u)\ge-K$ and the mean value theorem gives the desired estimate:
$$
(f(u),(-\Dx)^{1/2}u)\ge -K\|(-\Dx)^{1/4}u\|^2_{L^2}
$$
which completes the proof of the theorem in the case of periodic boundary conditions.
\end{proof}
With the proved extra regularity of energy solutions, their uniqueness and further regularity can be obtained in a standard way, so we omit the technicalities and only remind below the definition of the global attractor followed by the statement of the main result, proved in \cite{ASqdw}.

 \begin{Def}
{\rm A compact subset $\Cal A$ of a Banach space $\Cal E$ is called a {global attractor} for a semigroup $S(t):\Cal E\to\Cal E$, if}
\begin{enumerate}
\item {\rm $\Cal A$ is strictly invariant, i.e. $S(t)\Cal A=\Cal A$ $\forall t\geq 0$};
\item {\rm for any bounded set $B\subset \Cal E$ and any neighbourhood $\Cal{O}(\Cal A)$ of the set $\Cal A$, there exists a time $T=T(B,\Cal{O})$ such that}
$
S(t)B\subset\Cal{O}(\Cal A)
$
{\rm for all $t\geq T$}.
\end{enumerate}
\end{Def}
\begin{The}\label{Th1.attr} Let the assumptions of Theorem \ref{th main1} hold. Then, the energy solution of problem \eqref{eq dw} is unique and the solution semigroup $S(t)$ associated with problem \eqref{eq dw} in the phase space $\Cal E$ possesses a global attractor $\Cal A$ which is bounded in the more regular space $\Cal E_1:=[H^2(\Omega)\cap H^1_0(\Omega)]\times H^1_0(\Omega)$.
\end{The}
 \begin{Rem} {\rm The existence of exponential attractor bounded in $\Cal E_1$ that, in particular, implies finite fractal dimension of the global attractor $\Cal A$  is also verified in \cite{ASqdw}}.
\end{Rem}

\section{Quintic wave equation: the case of weak damping $\theta=0$}\label{sec4}
The extra space-time regularity of energy solutions is based here on the following non-trivial result concerning Strichartz type estimates for the linear wave equation in a bounded domain, see  \cite{sogge1, stri}.
\begin{Lem}
\label{Th LinStr}
Let $\xi_0\in\Cal E$, $G(t)\in L^1([0,T];L^2(\Omega))$ and $v$ solves the linear equation
\begin{equation}
\Dt^2 v-\Delta v =G(t),\quad v|_{\partial{\Omega}}=0,\ \xi_v|_{t=0}=\xi_0.
\end{equation}
 Then $v\in L^4([0,T];L^{12}(\Omega))$ and the following estimate holds:
\begin{equation}
\|v\|_{L^4([0,T];L^{12}(\Omega))}\leq C_T(\|\xi_0\|_\Cal{E}+\|G\|_{L^1([0,T];L^2(\Omega))}),
\end{equation}
where $C$ may depend on $T$ but it is independent of $\xi_0$ and $G$.
\end{Lem}
However, in contrast to the previous case, we do not know whether or not {\it all} energy solutions possess this extra regularity. For this reason, we restrict ourselves to consider only such energy solutions $u(t)$ of \eqref{eq dw} which belong to the space $L^4(0,T;L^{12}(\Omega))$. We will refer in the sequel to such solutions as {\it Shatah-Struwe} solutions of problem \eqref{eq dw}.
\par
The next theorem which gives the global well-posedness of problem \eqref{eq dw} in the class of Shatah-Struwe solutions can be proved with the help of the so-called Pohozhaev-Morawetz identity and the above mentioned Strichartz estimate, see \cite{stri}.

\begin{The}\label{Th.exist} Let $\theta=0$, $g\in L^2(\Omega)$ and the nonlinearity $f$ satisfy \eqref{0.fext} with $q=4$ as well as the following extra assumptions:
$$
1. \ |f''(u)|\le C(1+|u|^3),\ \ 2.\ \  f(u)u-4F(u)\ge-C.
$$
Then, for any $(u_0,u_0')\in\Cal E$, there exists a unique Shatah-Struwe solution $u(t)$ of problem \eqref{eq dw} defined for all $t\ge0$.
\end{The}
Thus, the (Shatah-Struwe) solution semigroup $S(t):\Cal E\to\Cal E$ associated with problem \eqref{eq dw} is well-defined. Moreover, these Shatah-Struwe solutions have  a number of good properties: they satisfy the energy identity \eqref{0.e} and the dissipative estimate \eqref{0.enest}; they are more regular, say, $\xi_u(t)\in\Cal E_1$ for all $t$ if $\xi_u(0)\in\Cal E_1$, etc.
\par
However, in contrast to the previous case, we do not know whether or not the analogue of estimate \eqref{1.extra} holds for the Strichartz norm $\|u\|_{L^4(t,t+1;L^{12}(\Omega))}$ and the previous theorem actually does not give any control of this norm as $t\to\infty$. Thus, the control of the Strichartz norm may be a priori lost when passing to the limit $t\to\infty$ and even if we initially consider the Shatah-Struwe solutions only, the other types of energy solutions (for which we have neither energy equality nor the uniqueness theorem) may a priori appear on the attractor.
\par
For this reason, despite the fact that the considered Shatah-Struwe solutions are unique, as an intermediate step, we need to exploit the existence of a weak attractor in the class of energy solutions where the uniqueness theorem is not known. Namely, as proved in \cite{ZelDCDS}, if we restrict ourselves to consider the energy solutions which can be obtained as a limit of Galerkin approximations only, then the trajectory dynamical system associated with these solutions possesses a trajectory attractor in the weak-star topology of $L^\infty_{loc}(\R_+,\Cal E)$ and (which is almost the same), the {\it multi-valued} semigroup $\bar S(t)$ associated with these solutions (the uniqueness is not known for that type of solutions) possesses a global attractor $\Cal A_w$ in a weak topology of~$\Cal E$, see also \cite{CV} for more details. Moreover, as shown in \cite{ZelDCDS} (see also \cite{Zel1,Zel2}),
the solutions on the attractor $\Cal A_w$ possesses the following backward regularity.

\begin{The}\label{Th2.weak} Under the above assumptions, the global attractor $\Cal A_w$ is generated by complete energy solutions $u(t)$, $t\in\R$ which are bounded in $\Cal E$ for all $t\in\R$. Moreover, for any such solution $u$, there exists $T=T(u)$ such that
$u(t)\in\Cal E_1$, for  $t\le -T$
and
$$
\|u\|_{L^\infty(-\infty,-T;\Cal E_1)}\le C,
$$
where the constant $C$ is independent of $u$.
\end{The}
Combining the result of Theorem \ref{Th.exist} with the above backward regularity is enough to verify the analogue of Theorem \ref{Th1.attr} for that case as well.
\begin{The}Let the assumptions of Theorem \ref{Th.exist} holds. Then the (Shatah-Struwe) solution semigroup $S(t)$ associated with problem \eqref{eq dw} possesses a global attractor $\Cal A$ in $\Cal E$ which is a bounded set in $\Cal E_1$.
\end{The}
Indeed, combining the aforementioned backward regularity, the fact that for any initial data $\xi_0\in\Cal E_1$, the corresponding Shatah-Struwe solution remains in $\Cal E_1$ and the proper version of weak-strong uniqueness, we establish that the weak attractor $\Cal A_w\subset \Cal E_1$, see \cite{KSZ} for more details. Thus, the energy equality holds for any solution belonging to this attractor. The asymptotic compactness of the solution semigroup in $\Cal E$ can be then established using the so-called energy method, see \cite{ball,rosa}. Finally, the additional regularity of the global attractor $\Cal A$ is based on this asymptotic compactness using more or less standard bootstrapping arguments, see \cite{KSZ}.

\begin{Rem}
\label{rem1 sub}
{\rm In the {\it subcritical} case $q<4$, the following analogue of the dissipative estimate for the Strichartz norm holds:}
\begin{equation}
\|u\|_{L^4([t,t+1];L^{12}(\Omega))}\leq Q(\|\xi_0\|_\Cal E)e^{-\alpha t}+Q(\|g\|),
\end{equation}
{\rm where the monotone function $Q$ is independent of $t$ and $u$ and the attractor theory in this case is essentially simpler, see \cite{KSZ}.}
\end{Rem}

\nopagebreak
{\small
{\em Authors' addresses}:\nopagebreak\\\nopagebreak
{\em Anton Savostianov}, University of Surrey, Guildford, UK,\nopagebreak\\\nopagebreak
 e-mail: \texttt{a.savostianov@\allowbreak surrey.ac.uk}\nopagebreak\\\nopagebreak
 {\em Sergey Zelik}, University of Surrey, Guildford, UK,\nopagebreak\\\nopagebreak
 e-mail: \texttt{s.zelik@\allowbreak surrey.ac.uk}
}

\end{document}